\documentclass[11pt,reqno]{amsart}
\usepackage{amsthm,amsmath,amssymb}
\usepackage{graphicx}
\usepackage{float}
\usepackage[colorlinks=true,
linkcolor=blue,
urlcolor=blue,
citecolor=blue]{hyperref}
\usepackage{mathtools}
\usepackage{relsize}
\usepackage{tikz}
\usepackage{enumerate}
\usepackage[toc,page]{appendix}
\usepackage{pdflscape}
\usepackage{multirow}
\usepackage{adjustbox,expl3,etoolbox} 
\usepackage[margin = 3.2cm]{geometry}
\usepackage{mathrsfs}
\usepackage{rotating}
\usepackage{times}
\usepackage{mathpple}

\newtheorem{thm}{Theorem}[section]

\theoremstyle{definition}

\newtheorem{cor}[thm]{Corollary}
\newtheorem{exa}[thm]{Example}

\begin{document}
	
\title[The Isolating Fusion Algorithm]{Applications of the Isolating Fusion Algorithm to Table Algebras and Association Schemes}
	
\author[Allen Herman]{Allen Herman\textsuperscript{1,*}}
\thanks{\textsuperscript{*}The work of the first author is supported by an NSERC Discovery Grant.} 
\thanks{\textsuperscript{1} Department of Mathematics and Statistics, University of Regina, Regina, Saskatchewan S4S 0A2, Canada}
\email{Allen.Herman@uregina.ca}
	
{\author[R. Maleki]{Roghayeh Maleki\textsuperscript{2,3}}
\thanks{\textsuperscript{2} University of Primorska, UP IAM, Muzejski trg 2, 6000 Koper, Slovenia.}
	\thanks{\textsuperscript{3} University of Primorska, UP FAMNIT, Glagolja\v ska 8, 6000 Koper, Slovenia.}
\email{roghayeh.maleki@famnit.upr.si}}

\begin{abstract} 
Let $\mathbf{B}$ be a basis for an $r$-dimensional algebra $A$ over a field or commutative ring with unity.  The semifusions of $\mathbf{B}$ are the partitions of $\mathbf{B}$ whose characteristic functions form the basis of a subalgebra of $A$, and fusions are semifusions that respect a given involution on $A$.  

 In this paper, we give an algorithm for computing a minimal semifusion (or fusion) of $\mathbf{B}$ that isolates a prescribed list of disjoint sums of basis elements of $\mathbf{B}$, when such a semifusion (or fusion) exists.  We apply this algorithm to three problems: (1) computing the fusion lattices for small association schemes of a given order; (2) producing explicit realizations of association schemes with transitive automorphism groups; and (3) producing examples of non-Schurian fusions of Schurian association schemes whose adjacency matrices have noncyclotomic eigenvalues.  The latter is of interest to the open question asking whether association schemes with transitive automorphsm groups can have noncyclotomic character values.   
\end{abstract}


\date{May 5, 2022 (updated June 15, 2023)}

\maketitle

\section{Introduction}
\vspace{0.5cm}
Let $R$ be a commutative ring with $1$, and let $I=\{0,1,\dots,r-1\}$.  An $R$-based algebra is a pair $(A,\mathbf{B})$ where $A=R\mathbf{B}$ is a free $R$-algebra with finite basis $\mathbf{B}=\{b_{0}=1_A, b_1,\ldots, b_{r-1}\}$.  In other words, for all of the products $b_{i}b_{j}=\sum_{k\in I}^{}\lambda_{ijk}b_{k} \in A$, where $b_i, b_j \in \mathbf{B}$, the structure constants $\lambda_{ijk}$ for $0 \le i,j,k \le r-1$ all lie in $R$.  Typical examples of based algebras are the group algebras of finite groups, adjacency algebras of finite association schemes or coherent configurations, table algebras, $C$-algebras, and reality-based algebras (see \cite{Higman89}, \cite{Blau09}, or \cite{Zieschang2}).  In any of these cases, $1_A$ is either an element of $\mathbf{B}$ or the sum of orthogonal idempotent elements of $\mathbf{B}$, and $A$ comes equipped with a $\mathbf{B}$-invariant $R$-linear involution $*$.  

Given an $R$-based algebra $(A,\mathbf{B})$ and any partition $\mathcal{P} = \{ I_1, I_2, \dots, I_s \}$ of the indexing set $I = \{0,1,\dots,r-1\}$ for the elements of $\mathbf{B}$, the {\it characteristic functions of the partition} $\mathcal{P}$ are the elements $b_{I_j} = b_{I_{j,1}} + \dots + b_{I_{j,k_j}}$, where $I_j = \{ I_{j,1}, \dots, I_{j,k_j} \}$ for $j=1,\dots,s$. 
We say that $\mathbf{D}$ is a {\it semifusion} of $\mathbf{B}$ if $\mathbf{D}$ is the set of characteristic functions of a partition $\mathcal{P}$ and the $R$-span of $\mathbf{D}$ is a free $R$-algebra of rank $|\mathbf{D}| = |\mathcal{P}|$.  When $(A,\mathbf{B})$ is a based algebra with involution whose distinguished basis $\mathbf{B}$ contains $1_A=b_0$, then we say that a semifusion $\mathbf{D}$ is a {\it fusion} of $\mathbf{B}$ when $\mathbf{D}$ is $*$-invariant and $b_0 \in \mathbf{D}$. The {\it trivial fusion} of such a based algebra $\mathbf{B}$ is $\{b_0, b_1+\dots+b_{r-1}\}$. 

The fusions of the basis $\mathbf{B}$ naturally form a lattice, with minimal element $b_{0}$ and maximal element $\mathbf{B}$.  The size of $\mathbf{B}$ is the main obstruction to the calculation of its fusion lattice, as testing the partitions of $\mathbf{B}$ one-by-one is not practical for bases of size much higher than $10$.  Another well-known approach to finding fusions, an application of the Bannai-Muzychuk character table criterion (\cite{bannai1991subschemes},\cite{muzychuk1987v}), can be applied in the case of commutative table algebras, but for searches it has similar issues with partitions.  So this article is motivated by the need for methods to generate fusions that can be applied to larger bases.  This is in part due to progress in the classification of small association schemes, which is now up to $34$ \cite{Hanaki-Miyamoto}.  Indeed, the first application of the isolating fusion algorithm we present here finishes the calculation of fusion lattices for association schemes of a given order for all orders in the classification (see \cite{HMprogram}).  The second application is an efficient way to obtain explicit realizations of an association schemes as fusions of two-orbit configurations.  This is especially effective for constructing schemes with transitive automorphism groups $G$ as fusions of the Schurian association scheme of $G$ relative to a vertex stabilizing subgroup.  Another motivation for this work has been investigations concerning the cyclotomic eigenvalue and character value questions, which ask whether or not the values of irreducible characters of association schemes have to lie in a cyclotomic number field (see \cite[pg. 123]{bannaialgebraic} and \cite{Herman-Survey}).  Examples of noncommutative association schemes with noncyclotomic eigenvalues had been noted earlier, but these examples were Schurian.   In the last section we apply the isolating fusion algorithm to produce some examples of non-Schurian schemes with transitive automorphism groups that have adjacency matrices with noncyclotomic eigenvalues.  

Our implementation of the isolating fusion algorithm uses GAP \cite{GAP4}, the code and fusion lattice data has been posted to Github \cite{HMprogram}.  The GAP packages {\tt coco2p} \cite{coco2p} and {\tt AssociationSchemes} \cite{AssociationSchemes} both offer advanced functions for calculating fusions.  In {\tt coco2p}, the fusions of coherent configurations are calculated using either partitions or graph coloring techniques.   The fusion functions in {\tt AssociationSchemes} versions 2.0$^+$ for intersection algebras have the same theoretical foundation as our isolating fusion algorithm, and were developed independently at about the same time.  {\tt AssociationSchemes} also offers functions that enable calculation of the fusion lattice of association schemes of a given order.  (We would like to thank Jesse Lansdown for updating us on the latest capabilities of the {\tt AssociationSchemes} package.) 

\section{The Isolating Fusion Algorithm}

In this section we give a condition on the structure constants of $\mathbf{B}$ that, for a given subset $I\subseteq \{0,1,\dots, r-1\}$, enables us to directly compute a minimal semifusion of $\mathbf{B}$ that contains the characteristic function $b_I$ of $I$.  We call such a semifusion a {\it $b_I$-isolating semifusion}.  If $I$ contains a single index i.e., $I = \{i\}$ for all $i\in I\setminus \{0\}$, then we call it a {\it $b_i$-isolating semifusion}.  When $|I|=k$, we also refer to $b_I$-isolating semifusions as $k$-isolating semifusions.

We will assume the basis $\mathbf{B}$ always has $b_0=1_A$, even though this is not essential to the method.  Every table algebra has the trivial fusion $\{b_{0},b_{1}+\dots+b_{r-1}\}$, which is the unique minimal $b_{0}$-isolating fusion.  

Our $b_I$-isolating semifusion criterion is based on a theorem which generalizes a well-known and elementary observation concerning fusions of association schemes (see \cite[(1.3)(ii)]{Bannai-Song1993}). 

\begin{thm}
	Let $\mathbf{B}=\{b_0=1_A,b_1,b_2,...,b_{r-1}\}$ be a basis of a free $R$-algebra containing $1_A$ with structure constants given by $\lambda_{ijk}$ for $i, j, k \in \{ 0, 1, \dots, r-1 \}$. Suppose $\mathbf{D}$ is a semifusion of $\mathbf{B}$, and that $b_I,b_J,b_K\in \mathbf{D}$, where $I,J$, and $K$ are subsets of $\{1,...,r-1\}$ with $I=\{i_1,...,i_{|I|}\}$, $J=\{j_1,...,j_{|J|}\}$, and $K=\{k_1,...,k_{|K|}\}$. Let the structure constants of $\mathbf{D}$ be given by $\{\lambda_{IJK}: b_{I},b_{J},b_{K}\in \mathbf{D}\}$, where $b_{I}b_{J}=\sum_{b_{K}\in \mathbf{D}}^{}\lambda_{IJK}b_{K}$. Then, for any $k\in K$,
	\begin{align*}
		\lambda_{IJK}=\sum_{i\in I} \sum_{j\in J} \lambda_{ijk}
	\end{align*}
	is constant.
	\label{3.1}
\end{thm}
\begin{proof} We have
	\begin{align*}
		b_{I}b_{J}&={\huge \left (\sum_{i\in I} b_{i} \right) }\huge \left (\sum_{j\in J}^{}b_{j}\right) =\sum_{k=0}^{r-1}\sum_{i\in I}^{}\sum_{j\in J}^{}\lambda_{ijk}b_{k} =\sum_{b_{K}\in \mathbf{D}}^{}\sum_{k\in K}^{}\sum_{i\in I}^{}\sum_{j\in J}^{}\lambda_{ijk}b_{k}\\
		&=\sum_{b_{K}\in \mathbf{D}}^{}\huge \left (\sum_{i\in I}^{}\sum_{j\in J}^{}\lambda_{ijk'} \left (\sum_{k\in K}^{}b_{k}\right )\right) {\normalsize \,\,\, (\mbox{for any } k'\in K, \mbox{ since } b_{K}=\sum_{k\in K}^{}b_{k}\in \mathbf{D})}       \\
		&=\sum_{b_{K}\in \mathbf{D}}^{}\huge \left(\sum_{i\in I}^{}\sum_{j\in J}^{}\lambda_{ijk'}\right)b_{K}.
	\end{align*}
\end{proof}	

\begin{cor} 
	For any $K\subseteq \{1,2,...,r-1\}$, $\mathbf{D}$ is a $|K|$-isolating semifusion of $\mathbf{B}$ if and only if for all $b_{I}, b_{J}\in \mathbf{D}$;\\
	(i) $\lambda_{IJK}=\sum_{i\in I}^{}\sum_{j\in J}^{}\lambda_{ijk}$ is constant for every $k\in K$\\
	(ii) $\lambda_{IKJ}=\sum_{i\in I}^{}\sum_{k\in K}^{}\lambda_{ikj}$ is constant for every $j\in J$; and\\
	(iii) $\lambda_{KIJ}=\sum_{k\in K}^{}\sum_{i\in I}^{}\lambda_{kij}$ is constant for every $j\in J$.
\end{cor}

This corollary suggests an algorithm that can compute a {\it minimal $\mathcal{I}$-isolating semifusion}.  Given a collection $\mathcal{I} = \{I_1,\dots,I_q\}$ of disjoint subsets of the index set $\{0,\dots,r-1\}$ of the basis $\mathbf{B} = \{b_0,b_1,\dots,b_{r-1} \}$, it outputs a minimal semifusion (or fusion) $\mathbf{F}$ for which $b_{I_1}, \dots, b_{I_q}$ are basis elements of $\mathbf{F}$, if such a semifusion (or fusion) exists.  Here we outline the steps under the extra assumption $b_0 \in \mathbf{F}$, which would be the case for the fusion of a table algebra.\\

\begin{enumerate}[{\bf Step~1.}]
	\item  Let $[r-1]=\{1,\dots,r-1\}$.  Suppose $I_1,\dots,I_q$ are disjoint subsets of $[r-1]$, and let $\mathcal{I}=\{I_1,\dots,I_q\}$.  Start with the partition 
$$\Sigma_0 = \{ I_1, \dots, I_q, [r-1]-(I_1 \cup \dots \cup I_q) \}.$$  For $i,j=1,\dots,q$, compute the expression of $(b_{I_i}b_{I_j})$ in terms of $\mathbf{B}$, and determine the smallest collection $\Sigma_1$ of subsets of $[r-1]$ for which $(b_{I_i}b_{I_j}) \in span_{\mathbb{C}}( \{b_0 \} \cup \{b_J : J \in \Sigma_1 \})$.  If any of the $b_{I_j}$ for $j=1,\dots,q$ do not lie in $span_{\mathbb{C}}(\{b_0 \} \cup \{b_J : J \in \Sigma_1 \})$, the algorithm returns {\tt fail}.  In this case, there is no $\mathcal{I}$-isolating fusion. (In practice, we can continue and the algorithm will return the first refinement of $\mathcal{I}$ it finds that produces a fusion).  Otherwise, Step 1 outputs $\Sigma_1$. \\

\item Let $\Sigma_1 = \{J_1=I_1,\dots,J_q=I_q,J_{q+1},\dots,J_d\}$.  Find the maximal refinement $\Sigma_2$ of $\Sigma_1$ for which $b_{J_i} b_{J_j} \in span_{\mathbb{C}}(\{b_0\} \cup \{b_J : J \in \Sigma\})$ for all pairs $J_i,J_j \in \Sigma_1$ with $1 \le i \ne j \le d$.  If $b_{I_k} \notin span_{\mathbb{C}}(\{b_0\} \cup \{b_J : J \in \Sigma_{i,j} \})$ for some $k = 1,\dots,q$, the algorithm returns {\tt fail}.  If {\tt fail} is not returned, and $\Sigma_2$ is a proper refinement of $\Sigma_1$, replace $\Sigma_1$ by $\Sigma_2$ and repeat Step 2.  Eventually, either {\tt fail} will be returned at point or the algorithm stops with $\Sigma_2 = \Sigma_1$.  In the latter case the output is $\mathbf{F} = \{b_0\} \cup \{b_J : J \in \Sigma_2\}$, a minimal $\mathcal{I}$-isolating semifusion of $\mathbf{B}$. \\ 

\item  Apply this extra step only if $\mathbf{B}$ is invariant under an involution of $A$, and an $\mathcal{I}$-isolating fusion is desired.  First, start with a minimal partition $\Sigma'_0$ for which $b_{I_1},\dots,b_{I_q},b_{I_1}^*, \dots,b_{I_q}^* \in \{b_J : J \in \Sigma'_0 \}$.  If no such partition exists, return {\tt fail}.  Repeat Steps 1 and 2.  If {\tt fail} is not returned, and Step 2 outputs a semifusion that is not a fusion, then we refine the resulting $\Sigma_2$ to the minimal partition $\Sigma'_2$ for which $b_{J_i}, b_{J_i}^* \in \{b_J : J \in \Sigma'_2 \}$.  We then replace $\Sigma'_0$ by $\Sigma'_2$ and repeat Steps 1 and 2 again.  Repeat this procedure until either {\tt fail} is returned or the output of Step 2 is a fusion. 
\end{enumerate}

While the algorithm is reasonably simple to implement, its main advantage is that isolating fusions are calculated directly, so it can be applied with success on reasonably large bases, and works equally well in both commutative and noncommutative situations.  For instance, in the course of generating data for the last section here we were applying it to obtain singleton-isolating fusions of Schurian schemes of order between $500$ and $600$ within a few minutes.  To illustrate the use of the algorithm for association schemes and coherent configurations, we give applications in the next sections. 

\section{Calculating fusion lattices for association schemes of a given order}

We have used the isolating fusion algorithm to calculate complete fusion lattices for association schemes of a given small order, up to and including $30$.  This had been attempted earlier at the time of the classification, but the first author noticed the fusion lattices for orders $24$, $28$, and $30$ were not complete; some fusions had not been detected.  The new fusion lattices are recorded at \cite{HMprogram}.   

The main advantage of this method for a scheme of rank $r$ is that the algorithm can produce several fusions quickly from its output starting from singleton subsets of $\{1,\dots,r-1\}$ of size $1$, $2$, $3$, etc., (to the extent the machine allows), rather than searching partitions.   Fusions of larger rank containing these single-set isolating fusions are found using the multiple-set isolating option for the algorithm.  (In practice, finding all $1$-isolating fusions can be done efficiently even when the rank $r$ is in the hundreds, for seeds of larger size there tends to be too many in these cases for a complete search to finish, but we can sample randomly generated seeds to reliably search for fusions.) 

As the association schemes of a given order are classified up to combinatorial automorphism and not algebraic automorphism, each minimal isolating fusion that our algorithm produces must be combinatorially identified.  This is achieved with some GAP functions developed earlier by Herman and Sikdar \cite{HermanWebsite}, which find a permutation matrix that matches the set of standard matrices of the association scheme to the set of standard matrices in the classification of schemes of the given order.  To speed up the identification of the fusion scheme we first match the sets of minimal polynomials produced by our basis elements.  This determines the scheme up to the character table, and in cases where more than one scheme has this character table we must then look for matching sets of adjacency matrices. 

\begin{exa}
In this example we consider the fusions of the association scheme labeled {\tt as28no176 } in the classification \cite{Hanaki-Miyamoto}.  This quasi-thin non-Schurian association scheme of order $28$ and rank $16$ was studied extensively in \cite{klin2007association}.  Here we reconsider the fusions of this association scheme using our implementation of the isolating fusion algorithm. the computer implementation of our method.  
	
We begin by running the isolating fusion algorithm on singleton subsets of $[15]$ of size $1$ to $7$, recording the partitions which produce new association schemes in the classification.  For {\tt as28no176}, this process takes our desktop computer about 1 hour.  The same association scheme in the classification can occur for different partitions; this is due to the action of the algebraic automorphism group on $\mathbf{B}$, which is not to be confused with the combinatorial automorphism group of the association scheme. Rather than computing this beforehand, we record all the partitions as we go along and use the information to determine the action of the algebraic automorphism group on $\mathbf{B}$, and record the number of different partitions giving the same scheme as an exponent.  

		$$
			\hspace*{-1.6cm}
			\begin{array}{l}
			|\mathcal{I}|=1: \\
			\\
			{\tt as28no2}: \{\{1\},\{2,3,4,5,6,7,8,9,10,11,12,13,14,15\}\}^3, \\
			{\tt as28no3}: \{ \{ 1, 2, 3 \}, \{ 4, 5, 6, 7, 8, 9, 10, 11, 12, 13, 14, 15 \} \}^1, \\
			{\tt as28no7}:  \{  \{ 1, 2, 3, 4, 5, 8, 9, 14, 15 \}, \{ 6, 7, 10, 11, 12, 13 \} \}^4, \\
			{\tt as28no10}: \{\{1\},\{2,3\},\{4,5,6,7,8,9,10,11,12,13,14,15\}\}^3, \\
			{\tt as28no59}:  \{ \{ 1, 2, 3 \}, \{ 4, 5, 8, 9, 12, 13 \}, \{ 6, 7, 10, 11, 14, 15 \} \}^4, \\
			{\tt as28no72}:  \{ \{ 1, 2, 3 \}, \{ 4, 7, 9, 10, 12, 14 \}, \{ 5, 6, 8, 11, 13, 15 \} \}^1, \\
			{\tt as28no75}:  \{ \{ 1, 4, 5, 12, 13 \}, \{ 2, 8, 9, 14, 15 \}, \{ 3, 6, 7, 10, 11 \} \}^4, \\
			{\tt as28no87}: \{ \{ 1, 2, 3 \}, \{ 4, 5, 6, 7\}, \{8, 9, 10, 11\}, \{12, 13, 14, 15 \} \}^1, \\
			{\tt as28no110}:  \{ \{ 1, 2, 3 \}, \{ 4, 9, 12 \}, \{ 5, 8, 13 \}, \{ 6, 11, 15 \}, \{ 7, 10, 14 \} \}^4, \\
			{\tt as28no113}: \{\{1\},\{2,3\},\{4,5,6,7\},\{8,9,10,11\},\{12,13\},\{14,15\}\}^3, \\
			{\tt as28no114}: \{\{1\},\{2,3\},\{4,5,6,7\},\{8,9,10,11\},\{12,15\},\{13,14\}\}^3,\\ 
        	{\tt as28no138}: \{ \{1,2,3\}, \{4,7\},\{5,6\},\{8,11\},\{9,10\},\{12,14\},\{13,15\} \}^1; 
        		\end{array}$$
         		$$\begin{array}{l}
			|\mathcal{I}=2|: \\
			\\ 	 
			{\tt as28no76}: \{ \{1\},\{2\},\{3\},\{4,5,6,7,8,9,10,11,12,13,14,15\} \}^1, \\
			{\tt as28no79}:  \{ \{ 1 \}, \{ 2, 3 \}, \{ 4, 7, 9, 10, 12, 14 \}, \{ 5, 6, 8, 11, 13, 15 \} \}^3, \\
			{\tt as28no95}: \{ \{ 1 \}, \{ 2, 3 \}, \{ 4, 5, 6, 7 \}, \{ 8, 9, 10, 11 \}, \{ 12, 13, 14, 15 \} \}^3, \\	
			{\tt as28no145}: \{ \{1 \}, \{ 2, 3 \}, \{ 4, 7 \}, \{ 5, 6 \}, \{ 8, 11 \}, \{ 9, 10 \}, \{ 12, 14 \}, \{ 13, 15 \} \}^3,\\
			{\tt as28no155}: \{ \{1\},\{2\},\{3\},\{4, 7\},\{ 5, 6 \}, \{ 8, 11 \}, \{ 9, 10 \}, \{ 12 \}, \{ 13 \}, \{ 14 \}, \{ 15 \} \}^3; 
          	\\
          	\\
			|\mathcal{I}=3|: \\
	      	\\
			{\tt as28no91}: \{ \{1\},\{2\},\{3\},\{4,7,9,10,12,14\},\{5,6,8,11,13,15\} \}^3, \\
			{\tt as28no111}:  \{ \{1\},\{2\},\{3\},\{4,5,6,7\},\{8,9,10,11\},\{12,13,14,15\} \}^3, \mbox{ and }\\
			{\tt as28no149}: \{ \{1\},\{2\},\{3\},\{4,7\},\{5,6\},\{8,11\},\{9,10\},\{12,14\},\{13,15\} \}^3. \\        	
		\end{array}$$

The uniqueness of the partition for occurrences of some of these fusion schemes, such as {\tt as28no72}, allow us to determine the orbits of the algebraic automorphism group of $\mathbf{B}$ in its action on $\mathbf{B}$. 	With a bit more work one can determine this algebraic automorphism group: it acts on the indices of the non-identity elements $\{b_1,\dots,b_{15}\}$ as the permutation group $\langle (1,3,2)(4,9,12)(5,8,13)(6,11,15)(7,10,14),(4,7)(5,6)(9,10)(8,11) \rangle$. 
This group has order $12$ and is isomorphic to the alternating group $A_4$.  
	
We use the multiple-subset version of the algorithm to ensure all fusions have been found.  This can be achieved because it suffices to only consider combinations of disjoint subsets up to algebraic isomorphism that occur in the list of single-set isolating fusions.  In practice, to check that all maximal fusions have been found, one needs to check for intermediate fusions between the singleton partition and the partitions recorded for the existing maximal fusions.  In this case the scheme {\tt as28no155} was found first, above {\tt as28no145}, and an examination of partitions lying between these two produced the intermediate fusion {\tt as28no149}.  Note that we can see from our partition data that {\tt as28no113} and {\tt as28no114} are also maximal fusions of {\tt as28no155}, this was not recorded earlier in \cite[pg.~2008]{klin2007association}. \\
\end{exa}

\section{Realizing fusions of two-orbit coherent configurations}

Initially, the authors developed the isolating fusion algorithm with the idea to search for association schemes with noncyclotomic eigenvalues, which will be discussed in the next section.  It was after seeing the algorithm in action that they realized it was also a useful tool for obtaining concrete realizations of association schemes or coherent configurations.  In fact, every coherent configuration of order $n$ is a fusion of the two-orbit configuration that arises from its combinatorial automorphism group.  Since the isolating fusion algorithm makes it possible to directly compute fusions of association schemes or coherent configurations of reasonably large rank with some efficiency, and the two-orbit configuration of a permutation group $G$ acting on $\{1,\dots,n\}$ is relatively easy to construct, it is natural to use the algorithm to construct many coherent configurations.  We have found that the isolating fusion program  can do reasonably efficient calculations (i.e. in less than a few minutes) of fusions when the rank of the initial  configuration is less than $200$, the most expensive step for larger rank being the generation of the left regular matrices of the basis, as these are $r \times r$ for a configuration of rank $r$.  

\begin{exa} {\rm  In the case $G=1 \le S_n$, the two-orbit configuration of $G$ is the full matrix configuration of rank $n^2$, whose adjacency matrices are the full set of matrix units $\{E_{i,j} : 1 \le i,j \le n\}$.  This is the most difficult case of order $n$ configurations for our algorithm to handle.  (If $n$ is greater than $15$ a subroutine to generate the left regular matrices corresponding to the full matrix units at the start would be needed.)  The full matrix configuration has two types of singleton-isolating fusions, $\{E_{i,i}\}$-isolating fusions for $1 \le i \le n$, and $\{E_{i,j}\}$-isolating fusions for $1 \le i,j \le n$.   The $\{E_{i,i}\}$-isolating fusions are 
$$ \{ \{E_{i,i} \}, \cup_{v \ne i} \{E_{i,v}\}, \cup_{u \ne i} \{E_{u,i}\}, I - \{ E_{i,i} \}, J - \cup_{u,v \ne i \\ u \ne v} \{E_{u,v}\} \}, $$ 
and the $\{E_{i,j}\}$-isolating fusions for $i \ne j$ are 
$$ \{ \{E_{i,i}\}, \{E_{i,j}\}, \{E_{j,i}\}, \{E_{j,j}\}, \cup_{u \ne i,j} \{ E_{u,j} \}, I - \{ E_{i,i}, E_{j,j} \}, \cup_{v \ne i,j } \{ E_{i,v} \}, J - \cup_{u,v \ne i,j \\ u \ne v} \{E_{u,v}\} \}. $$
} \end{exa}

As our initial implementation is originally designed to work with bases of left regular matrices whose first element is $I$, it needs some small adjustments to work correctly when the input is a configuration that is not an association scheme.  After these adjustments, we are able to obtain correct fusions of coherent configurations by using its multiple set option to the union of the set of all fibers with the set of other elements we wish to isolate in the input.  So to obtain the above we apply the multiple set isolating function to $\{ \cup_{i=1}^n \{E_{i,i}\}, \{E_{i,j} \} \}$.

\begin{exa}  The non-Schurian association scheme {\tt as28no176} was constructed in \cite{klin2007association} as a homogeneous fusion of the two-orbit configuration of its automorphism group, which is an elementary abelian subgroup of $S_n$ of order $8$ that has $7$ orbits of size $4$ on $\{1,\dots,28\}$.  One can obtain this scheme with our algorithm from the basic matrix of the two-orbit configuration of such a permutation group as an $\mathcal{I}$-isolating fusion.  To do this we can take $\mathcal{I}$ to be a collection of three carefully-chosen sets, each containing $7$ basis elements of the configuration.  The first two of these sets fuse to a symmetric permutation matrix of valency $1$, and the third set must fuse to an asymmetric adjacency matrix of valency $2$.  In agreement with the sensitive description of the fusing patterns in \cite{klin2007association}, we have found there is more than one pattern of these choices that will produce a quasithin association scheme of the same rank, one results in {\tt as28no176} and the other in its Schurian partner {\tt as28no175}.   
\end{exa}

If an association scheme has a transitive automorphism group $G$, then its two-orbit configuration is precisely the Schurian association scheme whose relations are the $2$-orbitals of the action of $G$ on the left cosets of the stabilizer $H$ of a vertex.  In this case the adjacency algebra is precisely the double coset algebra $\mathbb{C}[G/\!\!/H]$, so once regular matrices for this Schurian scheme are obtained, the isolating fusion algorithm can be directly applied. 

The most common way to search for and construct primitive association schemes is to consider the case where $H$ is a maximal subgroup of a group $G$ appearing in the Atlas of Finite Simple Groups.    Indeed, this is the way many primitive strongly regular graphs are described in \cite{Brouwer-VanMaldeghem2022}, and it is quite straightforward to duplicate many of their constructions with the isolating fusion algorithm.   We illustrate this capability with a typical example. 

\begin{exa} The $Aut(Sz(8))$ graph is a strongly regular graph on $560$ vertices whose primitive association scheme is the fusion of a Schurian scheme on the automorphism group of the simple group $Sz(8)$ in the Suzuki family (see \cite[\S 10.55]{Brouwer-VanMaldeghem2022}).   We can obtain this scheme in GAP by taking $G_0$ to be the group $Sz(8)$ from the Atlas, then letting $G_1$ be its holomorph $G_0 \rtimes Aut(G_0)$, and letting $G$ be the subgroup of $G_1$ whose generators correspond to those of $Aut(G_0)$.  With this approach $Aut(Sz(8))$ is realized as a permutation group in GAP, which speeds up the calculation of its maximal subgroups.  Up to conjugacy we find five maximal subgroups, here we list their indices and numbers of double cosets (which correspond to the order and rank of the ensuing Schurian scheme):  
\begin{table}[H]
	\begin{tabular}{|c|c|c|}
		\hline
		Maximal Subgroups & $|G/H|$ & $|G/\!\!/H|$ \\ \hline 
		$H_1$ & 3 & 3  \\ \hline
		$H_2$ & 65 & 2  \\ \hline
		$H_3$ & 560 & 7  \\ \hline
		$H_4$ & 1456 & 27  \\ \hline
		$H_5$ & 2080 & 59  \\ \hline	
	\end{tabular}
\end{table}
$H_1$ is a normal subgroup of index $3$, and $H_2$ corresponds to a well-known doubly transitive action of $Aut(Sz(8))$ on $65$ points.  When we apply the isolating fusion algorithm to find fusions of $G/\!\!/H_3$, we find it has a fusion of rank $3$ that corresponds precisely to the $Aut(Sz(8))$ graph scheme.  We are able to use isolating fusions to show $G/\!\!/H_4$ has no singleton isolating fusions, and a random search for isolating fusions of this scheme found no proper fusions.   For $G/\!\!/H_5$, we generate the regular matrices using new formulas (see \cite{Tout2017}, \cite{Alsairafi-Herman2021}) that calculate structure constants of Schurian schemes directly.  Again it has no singleton isolating fusions, and a random search for other isolating fusions did not find a proper fusion.  So this gives strong evidence the last two of these have no proper fusions.    
\end{exa}

\section{A search for small association schemes with noncyclotomic character values}  

Our main motivation for the development of the isolating fusion technique is the {\it cyclotomic eigenvalue problem}, which asks if the entries of character tables of commutative association schemes always lie in cyclotomic extensions of $\mathbb{Q}$ \cite{bannaialgebraic}.  In other words, this asks if adjacency matrix of every relation in a finite commutative association scheme should have cyclotomic eigenvalues.  In \cite{Herman-Survey}, the first author asked if there are adjacency matrices of relations in finite noncommutative association schemes that have noncyclotomic eigenvalues, and if so, if there were noncommutative association schemes with noncyclotomic fields of character values.  The latter is known to hold for Schurian association schemes by Morita theory (see \cite{herman2011schur}), fusions of commutative association schemes with cyclotomic eigenvalues by the Bannai-Muzychuk criterion (see \cite{Bannai-Song1993}), and for {\it orbit} Schur rings where the partition of the group $G$ consists of the orbits of some subgroup of $Aut(G)$ by \cite[Proposition 2.12]{DelvauxNauwelaerts1998}.  
 
Soon after \cite{Herman-Survey} appeared, it was pointed out that some examples of noncommutative association schemes with noncyclotomic eigenvalues were already appearing in the classification of small association schemes.  In particular {\tt as26no27}, with rank 10, and {\tt as26no31}, with rank 14, have this property.  Both of these examples are Schurian association schemes, so do satisfy the cyclotomic character value property.  If $(X,S)$ is the association scheme, $G$ is the (combinatorial) automorphism group of $S$ and $H$ is the stabilizer in $G$ of a fixed $x \in X$, called vertex stabilizer, then $(X,S)$ is a fusion of the Schurian (i.e. $2$-orbit)  association scheme $(G/H, G/\!\!/H)$.  When $H=1$, the adjacency algebra of the corresponding $2$-orbit association scheme is isomorphic to the group algebra of $G$, and its fusions are the Schur rings of the group $G$.  We note that {\tt as26no27} is a Schurian association scheme with automorphism group $G$ isomorphic to {\tt SmallGroup(78,1)} in GAP and vertex stabilizer $H$ of order $3$, and {\tt as26no31} is a Schurian association scheme with automorphism group $G$ isomorphic to {\tt SmallGroup(52,3)} in GAP and vertex stabilizer $H$ of order 2.  {\tt as26no27} has the additional property that some of the adjacency matrices with noncyclotomic eigenvalues are diagonalizable, a condition that would be necessary to hold in counterexamples to the cyclotomic eigenvalue problem.   

With this in mind the authors set out to use the isolating fusion algorithm to look for association schemes with transitive automorphism groups for which some adjacency matrices have noncyclotomic eigenvalues.  To do so, we ran random searches of Schurian association schemes $(G/H,G/\!\!/H)$ for which $|G|\le 200$, and $H = 1$ or a non-normal subgroup of index $<7$,  skipping groups of order $128$.  Any with noncyclotomic eigenvalues were checked for being Schurian, which is done by comparing to the Schurian scheme defined on their combinatorial automorphism group with respect to the first vertex stabilizer.  Calculation of the combinatorial automorphism group makes use of Soicher's GAP package {\tt grape} \cite{grape}, which makes use of Brendan McKay's {\tt nauty} program \cite{nauty}.  Among these, Schurian schemes with noncyclotomic eigenvalues occur with reasonable frequency (we found more than 250 groups for which this happens), and diagonalizable basis elements with noncyclotomic eigenvalues are much less frequent (this occurred for only about 20 groups).   What we found most surprising, however, is that we found {\it only one} group $G$, identified in GAP as {\tt SmallGroup(96,72)}, for which our algorithm detects non-Schurian fusions with noncyclotomic eigenvalues for two of its Schurian schemes $(G/H,G/\!\!/H)$.  Such groups are the only ones of interest in studying the cyclotomic character values question for association schemes with transitive automorphism groups. 

\begin{exa} {\rm Let $G$ be the group is identified by {\tt SmallGroup(96,72)}, which has the presentation
$$ G \simeq (C_4 \times C_4) \rtimes C_6 \simeq \langle x,y,z : x^4=y^4=z^6=1, x^z=x^2y, y^z=xy \rangle.$$ 
First, let $H = \langle z^3 \rangle$, a non-normal subgroup of order $2$.  The Schurian scheme $(G/H,G/\!\!/H)$ has order $48$ and rank $30$.  It has two non-Schurian fusions for which some of the basis elements have noncyclotomic elements.  The first has rank $12$, it occurs as the isolating fusion of a sum of three elements, of valencies $1$, $2$, and $2$.  As elements of $\mathbb{C}G$, the sum of these three elements is $\frac{1}{2}[(Hz^2x^2H)^+ + (Hz^2xH)^+ + (Hz^2x^3yH)^+]$.  The minimal polynomial of this element is divisible by $x^3-3$, so it has noncyclotomic eigenvalues.    

The second has rank $15$, it occurs as the isolating fusion of $\frac{1}{2}[(Hz^2H)^+ + (Hz^2xyH)^+]$, a sum of elements of valency $1$ and $2$.  The minimal polynomial of this element is divisible by $x^3+3$, which has a non-Abelian Galois group. 

However, the values of all irreducible characters of both of these fusions lie in $\mathbb{Q}(\zeta_3)$, the field obtained by adjoining a cube root of unity to $\mathbb{Q}$.  (We can check this directly with GAP.  For algebraic number fields $F$, if we find that $dim(Z(FS))$ is equal to the number of centrally primitive idempotents of $FS$, then the values of irreducible characters lie in $F$.)

Second, let $H = 1$. In this case the adjacency algebra of the Schurian scheme $(G, G)$ is the group algebra, so the order and rank are both equal to $96$.  For it random searches with our algorithm produces two non-Schurian fusions (i.e. Schur rings on $G$) that have basis elements with noncyclotomic eigenvalues.  The largest has rank $30$, it occurs as the isolating fusion of $z^5+z^5x^2+z^5x^3$, whose minimal polyomial is divisible by $x^3-3$ and $x^3+3$.  It has a non-Schurian fusion of rank $16$ that has an element whose minimal polynomial is divisible by $x^3-3$.  Again the values of the irreducible characters of both of these non-Schurian fusions lie in $\mathbb{Q}(\zeta_3)$. 
}\end{exa}

\end{document}